\documentclass[a4paper]{amsart}
\usepackage{amsmath,amsthm,amssymb,latexsym,epic,bbm,comment, mathrsfs}
\usepackage{graphicx,enumerate,stmaryrd,color}
\usepackage[all,2cell]{xy}
\xyoption{2cell}

\newtheorem{theorem}{Theorem}
\newtheorem{lemma}[theorem]{Lemma}
\newtheorem{corollary}[theorem]{Corollary}
\newtheorem{proposition}[theorem]{Proposition}
\newtheorem{conjecture}[theorem]{Conjecture}
\newtheorem{problem}[theorem]{Problem}
\newtheorem{example}[theorem]{Example}

\usepackage[all]{xy}
\usepackage[active]{srcltx}
\usepackage[parfill]{parskip}
\usepackage{graphicx,enumerate}
\usepackage[dvipsnames]{xcolor}
\usepackage{tikz}
\usepackage{fullpage}

\newcommand{\tto}{\twoheadrightarrow}

\font\sc=rsfs10
\newcommand{\cC}{\sc\mbox{C}\hspace{1.0pt}}
\newcommand{\cG}{\sc\mbox{G}\hspace{1.0pt}}
\newcommand{\cM}{\sc\mbox{M}\hspace{1.0pt}}
\newcommand{\cR}{\sc\mbox{R}\hspace{1.0pt}}
\newcommand{\cI}{\sc\mbox{I}\hspace{1.0pt}}
\newcommand{\cJ}{\sc\mbox{J}\hspace{1.0pt}}
\newcommand{\cS}{\sc\mbox{S}\hspace{1.0pt}}
\newcommand{\cT}{\sc\mbox{T}\hspace{1.0pt}}
\newcommand{\cD}{\sc\mbox{D}\hspace{1.0pt}}
\newcommand{\cL}{\sc\mbox{L}\hspace{1.0pt}}
\newcommand{\cP}{\sc\mbox{P}\hspace{1.0pt}}
\newcommand{\cA}{\sc\mbox{A}\hspace{1.0pt}}
\newcommand{\cB}{\sc\mbox{B}\hspace{1.0pt}}
\newcommand{\cU}{\sc\mbox{U}\hspace{1.0pt}}
\font\scc=rsfs7
\newcommand{\ccC}{\scc\mbox{C}\hspace{1.0pt}}
\newcommand{\ccP}{\scc\mbox{P}\hspace{1.0pt}}
\newcommand{\ccA}{\scc\mbox{A}\hspace{1.0pt}}
\newcommand{\ccJ}{\scc\mbox{J}\hspace{1.0pt}}
\newcommand{\ccS}{\scc\mbox{S}\hspace{1.0pt}}
\newcommand{\ccG}{\scc\mbox{G}\hspace{1.0pt}}

\begin{document}
\title[Essential order on stratified algebras]
{Essential orders on stratified algebras\\ with duality
and $\mathcal{S}$-subcategories in $\mathcal{O}$}

\author{Volodymyr Mazorchuk and Elin Persson Westin}

\begin{abstract}
We prove uniqueness of the essential order for stratified
algebras having simple preserving duality, generalizing 
a recent result of Coulembier for quasi-hereditary algebras.
We apply this to classify, up to equivalence, regular integral
blocks of $\mathcal{S}$-subcategories in the BGG category
$\mathcal{O}$. We also describe various homological invariants
of these blocks.
\end{abstract}

\maketitle



\section{Introduction and description of the results}\label{s1}

The notion of a {\em highest weight category} was 
introduced in \cite{CPS1} being motivated, among others, 
by the properties of the  BGG category $\mathcal{O}$
associated to a triangular decomposition of a semi-simple
complex Lie algebra, see \cite{BGG,Hu}. Highest weight categories
satisfying reasonable finiteness conditions are described by
{\em quasi-hereditary algebras}, see \cite{DR}.
Many examples 
and applications of these structures appeared in the literature 
over the last 35 years (a naive MathSciNet search for
the keyword ``quasi-hereditary algebra'' produces more than 600 results). 

A quasi-hereditary algebra is a structure consisting of a
finite dimensional algebra $A$ and a choice of a partial order on
the set of isomorphism classes of simple $A$-modules
(this datum is assumed to satisfy some conditions). This datum
is used to define the so-called {\em standard} $A$-modules which
play the crucial role in the quasi-hereditary structure.
Usually, there are many different choices for a partial order
making a given algebra $A$ quasi-hereditary (for example,
$A$ is hereditary if and only if it is quasi-hereditary for
{\em any} choice of a linear order).
The recent paper \cite{Co} established the very surprising
fact that, under the assumption of existence of a simple preserving
duality, the quasi-hereditary structure on $A$ is, essentially, {\em unique}.
More precisely, any flexibility in the order is unessential in the sense
that it does not affect the standard modules.

There are various generalizations of quasi-hereditary algebras.
One of them is that of {\em (standardly) stratified algebras}, introduced in \cite{CPS}.
Assuming the simple preserving duality, this leads to the class of
{\em properly stratified} algebras, introduced in \cite{Dl}, see also
\cite{Fr}. Appearance of these kinds of structures in Lie theory
was established in \cite{FKM,FKM2} with several application
proposed in \cite{MS}, see also \cite{FKS,BLW}.

The first main result of the present paper provides an analogue of 
the uniqueness result of \cite{Co} for properly stratified algebras,
see Theorem~\ref{thm2} in Subsection~\ref{s2.3}.

In \cite{Co}, the main application of its uniqueness result is the
classification, up to equivalence, of (indecomposable) blocks of
BGG category $\mathcal{O}$. Such blocks are always described by
quasi-hereditary algebras admitting a simple preserving duality.
The paper \cite{FKM} introduces a generalization of $\mathcal{O}$
called {\em $\mathcal{S}$-subcategories in $\mathcal{O}$} and shows
that these are usually described by properly stratified algebras.
The realization of $\mathcal{S}$-subcategories in $\mathcal{O}$ as
subcategories of $\mathfrak{p}$-presentable
modules in $\mathcal{O}$ was further studied in \cite{MS}.
$\mathcal{S}$-subcategories in $\mathcal{O}$ appear naturally when
studying the action of projective functors on blocks of $\mathcal{O}$,
see \cite{BG,MS2}. In particular, the action of projective functors on the
principal block of an $\mathcal{S}$-subcategory models the permutation
module (i.e. the induced trivial module) for the Weyl group of
$\mathfrak{g}$ with respect to the parabolic subgroup corresponding to 
$\mathfrak{p}$.
Compared to $\mathcal{O}$, a natural parameterization of blocks for
$\mathcal{S}$-subcategories in $\mathcal{O}$ has an additional 
independent parameter (with the original category 
$\mathcal{O}$ being the special case corresponding to a
trivial  value of that parameter).
We address the problem of classification of blocks of 
$\mathcal{S}$-subcategories in $\mathcal{O}$, up to equivalence,
but at the moment it seems quite difficult (the main result of
\cite{Co} can be viewed as a special case of this problem).
We discuss some partial results  in this direction in Subsections~\ref{s3.8}
and \ref{s3.9}. In particular, we use shuffling and twisting functors to establish 
two different types of derived equivalences
between blocks of $\mathcal{S}$-subcategories in $\mathcal{O}$ 
in Propositions~\ref{propnew72} and \ref{propnew73}.

In Theorem~\ref{thm6}, which can be found in Subsection~\ref{s3.3},
we solve another special case (the ``opposite'' one with respect to 
the one considered
in \cite{Co}). 
In more detail, we classify, up to equivalence, regular 
blocks of $\mathcal{S}$-subcategories in $\mathcal{O}$. 
In our proof of Theorem~\ref{thm6} we use both Theorem~\ref{thm2}
and the two main results of \cite{Co}.

In  Section~\ref{s4}, we also study homological properties of
regular blocks of $\mathcal{S}$-sub\-ca\-te\-go\-ries in $\mathcal{O}$
in the spirit of \cite{Ma1},
in particular, the finitistic dimension of these blocks
(finiteness of which was established already in \cite{MO}). 
We provide a description of objects of finite projective dimension
in Theorem~\ref{thm12} and use it to obtain an explicit formula for
the finitistic dimension (in terms of Lusztig's $\mathbf{a}$-function)
as well as explicit formulae for projective dimensions of 
indecomposable tilting and injective objects (and also a less explicit formula
for projective dimensions of standard objects).
For tilting and injective
objects, we use 
Theorem~\ref{thm12} to reduce the question to known results
for category $\mathcal{O}$ from \cite{Ma1,Ma2}.
For standard objects, we use Theorem~\ref{thm12} and the results of \cite{CM}.
\vspace{3mm} 
 
\noindent
{\bf Acknowledgments.} For the first author, the research was 
partially supported by the Swedish Research Council (2017-03704)
and G{\"o}ran Gustafssons Stiftelse. We thank Joel Dahne for help
with Sage computations and we thank Kevin Coulembier
for helpful comments.
We also thank the referee for helpful comments.

\section{Essential order on stratified algebras with duality}\label{s2}

\subsection{Stratified algebras}\label{s2.1}

Let $\Bbbk$ be 
a field and $A$ a finite 
dimensional (associative) $\Bbbk$-algebra. Let
$\{L_{\lambda}\,:\,\lambda\in\Lambda\}$ be a complete and 
irredundant list of representatives of isomorphism classes of 
simple $A$-modules and define $d_{\lambda}=\dim_\Bbbk\mathrm{End}_A(L_{\lambda})$. Note that if $\Bbbk$ is algebraically closed, then $d_\lambda=1$ for all $\lambda\in\Lambda$. Assume that $\Lambda$ is equipped with
a fixed partial order $\prec$. Denote by $A$-mod the category 
of all finite dimensional left $A$-modules.

Recall that, for two $A$-modules $M$ and $N$, the {\em trace}
of $M$ in $N$ is the submodule of $N$ spanned by the images of all
$A$-module homomorphisms from $M$ to $N$.

For $\lambda\in\Lambda$, denote by $P_{\lambda}$ the indecomposable projective 
cover of $L_{\lambda}$ and by $I_{\lambda}$ the indecomposable injective envelope
of $L_{\lambda}$. Further, denote by $\Delta_{\lambda}$ the quotient of 
$P_{\lambda}$ by the trace in $P_{\lambda}$ of all $P_{\mu}$, 
where $\mu\not\preceq \lambda$. Set
$\overline{\Delta}_{\lambda}$ to be the maximal quotient of $\Delta_{\lambda}$
such that $[\overline{\Delta}_{\lambda}:L_{\lambda}]=1$. Dually, denote by
$\nabla_{\lambda}$ the intersection of the kernels of all morphisms from
$I_{\lambda}$ to $I_{\mu}$, where $\mu\not\preceq \lambda$, and define 
$\overline{\nabla}_{\lambda}$
as the maximal submodule of $\nabla_{\lambda}$ such that 
$[\overline{\nabla}_{\lambda}:L_{\lambda}]=1$. The modules $\Delta_{\lambda}$,
$\overline{\Delta}_{\lambda}$, $\nabla_{\lambda}$ and 
$\overline{\nabla}_{\lambda}$
are called {\em standard}, {\em proper standard}, 
{\em costandard} and {\em proper costandard}, respectively.
If we need to emphasize the role of $\prec$ in the definition,
we will use the notation $\Delta^{\prec}_{\lambda}$ and similarly for
the others.

Recall from \cite{CPS,Fr}, that $(A,\prec)$ is called
{\em standardly stratified} provided that, for each $\lambda\in\Lambda$, 
the kernel of the surjection $P_{\lambda}\tto \Delta_{\lambda}$
has a filtration (called a {\em standard filtration}), 
each subquotient of which is isomorphic to
$\Delta_{\mu}$, for some $\lambda\prec\mu$. Equivalently, 
$A$ is standardly stratified provided that, for each $\lambda\in\Lambda$, 
the cokernel of the injection $\overline{\nabla}_{\lambda}
\hookrightarrow I_{\lambda}$
has a filtration (called a {\em proper costandard filtration}), 
each subquotient of which is isomorphic to
$\overline{\nabla}_{\mu}$, for some $\lambda\preceq\mu$, see \cite{Fr}. 

A {\em simple preserving duality} is a contravariant anti-equivalence  
$\star$ of $A$-mod which preserves the isomorphism classes of 
simple $A$-modules. If a stratified algebra $A$ has a simple preserving
duality, then $A^{\mathrm{op}}$ is also stratified, in particular,
$A$ is {\em properly stratified} in the sense of \cite{Dl}.

\subsection{Essential order}\label{s2.2}

Let $(A,\prec)$ be a standardly stratified algebra. Define a new
partial order $\prec_e$ on $\Lambda$, called the {\em essential order},
as the minimal partial order which contains all pairs 
$(\lambda,\mu)\in \Lambda\times \Lambda$ such that
\begin{itemize}
\item $L_{\lambda}$ is a composition subquotient of $\Delta_{\mu}$ or
\item $L_{\lambda}$ is a composition subquotient of $\overline{\nabla}_{\mu}$. 
\end{itemize}
Clearly, $\lambda\prec_e\mu$ implies $\lambda\prec\mu$.
Note that, by \cite[Subsection~4.1]{Fr}, standard and proper costandard modules
form homologically orthogonal families of modules. 
That is, 
$$
\mathcal{F}(\Delta)=\{M\in A\text{-}\mathrm{mod} \mid 
\mathrm{Ext}^i(M,\overline{\nabla}_\lambda)=0 \: 
\text{ for all } i>0, \lambda\in \Lambda\}
$$ 
and 
$$
\mathcal{F}(\overline{\nabla})=\{M\in A\text{-}\mathrm{mod} 
\mid \mathrm{Ext}^i(\Delta_\lambda,M)=0  \:
\text{ for all } i>0, \lambda\in \Lambda\},
$$ 
where $\mathcal{F}(\Delta)$ and $\mathcal{F}(\overline{\nabla})$ 
denote the full subcategories of modules with standard and proper 
costandard filtrations, respectively.  
Therefore the condition that
$L_{\lambda}$ is a composition subquotient of $\overline{\nabla}_{\mu}$
is equivalent to the condition that $\Delta_{\mu}$ appears as a subquotient
in a standard filtration of $P_{\lambda}$.

\begin{lemma}\label{lem1}
If $(A,\prec)$ is standardly stratified, then $(A,\prec_e)$ 
is standardly stratified. Moreover, the standard modules for
both structures coincide and also the proper costandard
modules for both structures coincide.
\end{lemma}

\begin{proof}
It is enough to prove that the standard modules for  $\prec_e$
coincide with the standard modules for $\prec$.
Indeed, if we can do that, then the ordering conditions for filtrations
are automatically satisfied by the definition of $\prec_e$.
Furthermore, the fact that the proper costandard modules are the same
follows from, e.g. \cite[Subsection~4.1]{Fr}, that is from 
the fact that 
standard and proper costandard modules 
form homologically orthogonal families and 
hence determine one another, as families of modules.

Since $\prec_e\subset \prec$, directly from the definition of 
standard modules, we get $\Delta_{\lambda}^{\prec}\tto
\Delta_{\lambda}^{\prec_e}$, for every $\lambda\in\Lambda$.
We want to show that this surjection is, in fact, an isomorphism.
We will prove the claim by induction on the cardinality of $\Lambda$.
In case $\vert\Lambda\vert=1$, both $\prec$ and $\prec_e$ coincide
with the equality relation and hence we have nothing to prove.

Let $\lambda$ be a maximal element in $(\Lambda,\prec)$. Then
$P_{\lambda}=\Delta_{\lambda}^\prec$ and hence, for any  $\mu\in\Lambda$
such that $[P_{\lambda}:L_{\mu}]\neq 0$, we have
$\mu\prec_e\lambda$. Therefore $P_{\lambda}=\Delta_{\lambda}^{\prec_e}$.

Consider now the quotient $B$ of $A$ modulo the ideal $\mathcal{I}$ 
given by the trace of $P_{\lambda}$ in $A$. 
The algebra $B$ is standardly
stratified with respect to the restriction of $\prec$ to 
$\Lambda\setminus\{\lambda\}$. 
Note that 
$\vert\Lambda\setminus\{\lambda\}\vert<\vert\Lambda\vert$.
Therefore we may now apply the
inductive assumption to the algebra $B$ and conclude that standard 
$B$-modules with respect to $\prec$ and $\prec_e$ coincide.
From the definitions, we see that all standard 
$B$-modules coincide with the corresponding standard $A$-modules.
This completes the proof of the induction and of the lemma.
\end{proof}

Recall that a standardly stratified algebra $(A,\prec)$ is called
{\em quasi-hereditary} provided that 
$\Delta_{\lambda}=\overline{\Delta}_{\lambda}$, for all
$\lambda\in\Lambda$, see \cite{CPS1,DR}. Equivalently, 
we may require the equality $\nabla_{\lambda}=\overline{\nabla}_{\lambda}$, 
for all $\lambda\in\Lambda$. For quasi-hereditary algebras,
the essential order is discussed in \cite[\S~1.2.4]{Co}.
Lemma~\ref{lem1} generalizes some of the claims in \cite[\S~1.2.4]{Co}.

\subsection{Uniqueness of the essential order for stratified
algebras with simple preserving duality}\label{s2.3}

The following result, which establishes uniqueness of 
the essential order for stratified algebras with simple
preserving duality, generalizes \cite[Theorem~2.1.1]{Co}
where the case of quasi-hereditary algebras was considered.

\begin{theorem}\label{thm2}
Let $(A,\prec^1)$ and $(A,\prec^2)$ be standardly stratified.
Assume that $A$-mod has a simple preserving duality.
Then $\prec^1_e=\prec^2_e$.
\end{theorem}

Note that, under the assumption of existence of a simple
preserving duality, each $\Delta_\mu$ has a filtration
with subquotients isomorphic to $\overline{\Delta}_\mu$, and
each $\nabla_\mu$ has a filtration
with subquotients isomorphic to $\overline{\nabla}_\mu$.
Therefore the condition that $L_\lambda$ is a simple subquotient of
$\Delta_\mu$ is equivalent to the condition that $L_\lambda$ is a
simple subquotient of $\overline{\Delta}_\mu$.

In order to prove this, we will need some notation and lemmata.
For $\lambda,\mu\in\Lambda$, we denote by $c_{\lambda,\mu}$
the composition multiplicity $[P_{\lambda}:L_{\mu}]$.

\begin{lemma}\label{lem3}
Let $(A,\prec)$ be standardly stratified with a simple 
preserving duality $\star$. Let $\lambda\in\Lambda$ be a maximal element
with respect to $\prec$. Then $c_{\lambda,\lambda}\leq 
c_{\lambda,\mu}$, for any $\mu\in\Lambda$ such that 
$c_{\lambda,\mu}\neq 0$.
\end{lemma}

\begin{proof}
Because of $\star$, the algebra $A$ is properly stratified. Therefore
the standard module $P_{\lambda}=\Delta_{\lambda}$ has a filtration with all subquotients
isomorphic to $\overline{\Delta}_{\lambda}$. 
From the definition of
$\overline{\Delta}_{\lambda}$ we see that this filtration has length
$c_{\lambda,\lambda}$. Therefore any $c_{\lambda,\mu}$ is an integer
multiple of $c_{\lambda,\lambda}$. The claim follows.
\end{proof}

\begin{lemma}\label{lem4}
Let $(A,\prec)$ be standardly stratified with a simple 
preserving duality $\star$. Let $\lambda\in\Lambda$ be 
such that $P_{\lambda}\neq\Delta_{\lambda}$. Then there exists
$\mu\in\Lambda$ such that $0<c_{\lambda,\mu}<c_{\lambda,\lambda}$.
\end{lemma}

\begin{proof}
Choose $\mu\in\Lambda$ maximal among all those $\nu$ for
which $\Delta_{\nu}$ appears as a subquotient of a standard
filtration of $P_{\lambda}$. 
Note that $\mu$ is also maximal among those $\nu$ such that $\overline{\Delta}_\nu$ appears as a subquotient of a proper standard filtration of $P_\lambda$.
Then $c_{\lambda,\mu}$ coincides
with the multiplicity of $\overline{\Delta}_{\mu}$ in a proper
standard filtration of $P_{\lambda}$. In particular, $c_{\lambda,\mu}\neq 0$.

By the BGG reciprocity, see e.g. \cite[Theorem~2.5]{ADL}, 
\cite[Lemma~12]{Fr}, and the fact that $A$ has a simple preserving 
duality, we have the following relation between the 
multiplicity of $\Delta_\mu$ in a standard filtration of 
$P_{\lambda}$ and the composition multiplicity  of $L_{\lambda}$
in $\overline{\Delta}_{\mu}$:
\begin{displaymath}
d_{\mu}[P_{\lambda}:\Delta_\mu]=d_{\lambda}[\overline{\Delta}_{\mu}:L_{\lambda}].
\end{displaymath}
(We note that \cite{Fr} only proves the result for algebras over an algebraically 
closed field, however, the general case, which takes $d_\lambda$
and $d_\mu$ into account, is straightforward.)
Consequently, the right hand side is non-zero and hence 
$[\overline{\Delta}_{\mu}:L_{\lambda}]\geq 1$.  
Therefore
\begin{align*}
c_{\lambda,\lambda}&\geq [P_{\lambda}:\overline{\Delta}_{\lambda}] [\overline{\Delta}_{\lambda}:L_{\lambda}]+[P_{\lambda}:\overline{\Delta}_{\mu}][\overline{\Delta}_{\mu}:L_{\lambda}] \\
&\geq 1+c_{\lambda,\mu}.
\end{align*}

This completes the proof.
\end{proof}

\begin{proof}[Proof of Theorem~\ref{thm2}.]
What we really need to prove is that the corresponding 
standard modules with  respect to $\prec^1_e$ and $\prec^2_e$ coincide.
If we prove that, the fact that $\prec^1_e$ and $\prec^2_e$ coincide
would follow directly from the definition of an essential order.
We proceed by induction on $\vert\Lambda\vert$. 
If $\Lambda$ is a singleton, then 
$\prec^1=\prec^2=\prec^1_e=\prec^2_e$ and we have nothing
to prove.

Lemmata~\ref{lem3} and \ref{lem4} together imply that the elements $\lambda\in\Lambda$ having the property that $c_{\lambda,\lambda}\leq c_{\lambda,\mu}$, for all $\mu$, are maximal
with respect to all the orders 
$\prec^1$, $\prec^2$, $\prec^1_e$ and $\prec^2_e$.
In particular, the corresponding standard module is projective
and hence is the same for both $\prec^1_e$ and $\prec^2_e$.

Recall that, for any projective module $P$, the trace of $P$ in the left regular $A$-module ${}_AA$ is a two-sided ideal of $A$. Let $\mathcal{I}$ be such an ideal given by the trace of $P_{\lambda}$ in ${}_AA$. Factoring this ideal out,
we can apply the inductive assumption to the stratified algebra
$A/\mathcal{I}$ and conclude that all corresponding standard modules with 
respect to $\prec^1_e$ and $\prec^2_e$ coincide. This completes the proof.
\end{proof}

\subsection{The weighted poset invariant}\label{s2.4}

Let $(A,\prec)$ be a standardly stratified algebra with a simple preserving duality.
Define the function $\nu:\Lambda\to\mathbb{Z}_{>0}$ as follows:
\begin{displaymath}
\nu(\lambda):=\dim_{\Bbbk}\mathrm{End}(\Delta_\lambda).
\end{displaymath}
From Theorem~\ref{thm2} it follows that the weighted poset
$(\Lambda,\prec_e,\nu)$ is an invariant of the Morita equivalence class of $A$.

Certainly, there are other natural choices for $\nu$ which give rise to similar invariants,
for example, one could use $[\Delta_\lambda:L_\lambda]$ instead of 
$\dim_{\Bbbk}\mathrm{End}(\Delta_\lambda)$ (the two numbers coincide provided that
$\Bbbk$ is algebraically closed). We have chosen $\dim_{\Bbbk}\mathrm{End}(\Delta_\lambda)$
as this number is more straightforward to compute in the situation where this invariant 
is used in Subsection~\ref{s3.7}.

\section{$\mathcal{S}$-subcategories in $\mathcal{O}$}\label{s3}

\subsection{Category $\mathcal{O}$}\label{s3.1}

Consider a simple finite dimensional Lie algebra $\mathfrak{g}$
over $\mathbb{C}$. Fix a triangular decomposition
\begin{displaymath}
\mathfrak{g}=\mathfrak{n}_-\oplus\mathfrak{h}\oplus\mathfrak{n}_+. 
\end{displaymath}
To this datum, we have the associated Bernstein-Gelfand-Gelfand
(BGG) category $\mathcal{O}$, see \cite{BGG,Hu}.
We refer the reader to \cite{BGG,Hu,Di,Ja} for further details on 
category $\mathcal{O}$.

Let $W$ be the Weyl group of $\mathfrak{g}$ and 
$\mathbf{R}$ be the corresponding root system in $\mathfrak{h}^*$,
with the basis $\pi$ corresponding to our choice of $\mathfrak{n}_+$.
We denote by $\rho$ the half of the sum of all positive roots.
We have the corresponding {\em dot-action} of $W$ on 
$\mathfrak{h}^*$ given by $w\cdot\lambda=w(\lambda+\rho)-\rho$.

Simple objects in $\mathcal{O}$ are the simple highest weight modules
$L(\lambda)$, where $\lambda\in\mathfrak{h}^*$. In particular,
we have infinitely many non-isomorphic simple objects. However, 
$\mathcal{O}$ decomposes into a direct sum of indecomposable
{\em blocks}, 
each of which has finitely many simples and is equivalent
to $A$-mod for some finite dimensional, associative and 
quasi-hereditary algebra $A$. In fact, if two simples $L(\lambda)$
and $L(\mu)$ belong to the same block, then $\lambda\in W\cdot \mu$.

Category $\mathcal{O}$ has a simple preserving duality, which we denote by $\star$. This duality is an involutive contravariant equivalence which fixes every isomorphism class of simple objects. We refer to  \cite[\S~3.2]{Hu} for details.

We denote by $\mathcal{O}_{\mathrm{int}}$ the full subcategory of
$\mathcal{O}$ consisting of modules with integral weights.
By \cite{So}, each block of $\mathcal{O}$ is equivalent to an
{\em integral block}, that is, to a block of $\mathcal{O}_{\mathrm{int}}$ 
for some (possibly different) Lie algebra $\mathfrak{g}$.
Using this we restrict our attention to $\mathcal{O}_{\mathrm{int}}$.
For $\mathcal{O}_{\mathrm{int}}$, two simples $L(\lambda)$
and $L(\mu)$ belong to the same block if and only if 
$\lambda\in W\cdot \mu$. To simplify our notation, if $\lambda$
is dominant and
clear from the context, we write $L_w$ instead of $L(w\cdot\lambda)$, for $w\in W$.

We denote by $P(\lambda)$ the indecomposable projective cover
of $L(\lambda)$ in $\mathcal{O}$ and by $I(\lambda)$ the corresponding
indecomposable injective envelope. We denote by $\Delta(\lambda)$
the standard module (also called Verma module) with top $L(\lambda)$ and
by $\nabla(\lambda)$ the corresponding costandard module.
Then $\nabla(\lambda)=\Delta(\lambda)^\star$, where $\star$
is the simple preserving duality on $\mathcal{O}$. 
We denote by $T(\lambda)$ the indecomposable tilting
module with highest weight $\lambda$.
Similarly to the above, we use the notation $P_{w}$ etc. when appropriate.

For a dominant integral weight $\lambda$, 
we denote by 
$\mathcal{O}_{\lambda}$ the block of $\mathcal{O}$ containing
$L(\lambda)$. Let $W_{\lambda}$ denote the {\em dot-stabilizer}
of $\lambda$. 
Then simple modules in $\mathcal{O}_{\lambda}$ are
naturally parameterized by the set of longest coset
representatives in $W/W_{\lambda}$. 
{\em Soergel's combinatorial 
description} for integral blocks of $\mathcal{O}$, see \cite{So}, 
implies that $\mathcal{O}_{\lambda}\cong \mathcal{O}_{\lambda'}$
if  $W_{\lambda}=W_{\lambda'}$ (note that the latter is a genuine
equality and not an isomorphism). This allows us to parameterize
blocks of $\mathcal{O}$ by pairs $(W,S)$, where $W$ is a Weyl
group and $S$ a parabolic subgroup of $W$ (i.e. the {\em singularity}
of a block). One of the main results in 
\cite{Co} asserts that the blocks given by $(W,S)$ and $(W',S')$
are equivalent if and only if the partially ordered sets 
$W/S$ and $W'/S'$, with respect to the Bruhat order
(more precisely, the restriction of the Bruhat order
to the set of the longest coset representatives), are isomorphic.

\subsection{$\mathcal{S}$-subcategories in $\mathcal{O}$}\label{s3.2}

In this paper we will be interested in the extension of the results of \cite{Co}
to the so-called $\mathcal{S}$-subcategories in $\mathcal{O}$,
introduced in \cite{FKM}, see also \cite[Section~2]{MS}. 
To define these we need to fix
another parabolic subgroup $G$ in $W$, completely unrelated to $S$.

So, we fix an integral dominant weight $\lambda$ with dot-stabilizer
$S$ 
and split $W\cdot \lambda$ into two sets: the set $X$ consisting
of all $\mu\in W\cdot \lambda$ such that $w\cdot \mu\geq \mu$, for any $w\in G$,
and the set $Y:=(W\cdot \lambda)\setminus X$. 

Let $\mathcal{C}$ denote the Serre subcategory of $\mathcal{O}_{\lambda}$
generated by all $L(\mu)$, where $\mu\in Y$. The {\em $\mathcal{S}$-subcategory}
of $\mathcal{O}$ associated to $\lambda$ and $G$ is the abelian quotient
\begin{displaymath}
\mathcal{S}_{\lambda,G}:=\mathcal{O}_{\lambda}/\mathcal{C}. 
\end{displaymath}
We denote by $\Phi:\mathcal{O}_{\lambda}\to \mathcal{S}_{\lambda,G}$
the corresponding exact quotient functor.
An alternative description of $\mathcal{S}_{\lambda,G}$ is as follows:
$\mathcal{S}_{\lambda,G}$ is equivalent to the category of $eAe$-modules,
where $A$ is the unique, up to isomorphism, basic associative algebra such that $\mathcal{O}_{\lambda}$ is equivalent to $A$-mod, 
and $e$ is the idempotent corresponding to all $\mu\in X$,
that is:
\begin{displaymath}
eAe\cong\mathrm{End}_{\mathcal{O}}
\big(\bigoplus_{\mu\in X}P(\mu)\big)^{\mathrm{op}}. 
\end{displaymath}
The algebra $eAe$ is, in general, no longer quasi-hereditary,
however, it is properly stratified, see \cite{FKM}.
Note that $\mathcal{S}_{\lambda,\{e\}}=\mathcal{O}_{\lambda}$.

Let $\mathfrak{p}$ be the parabolic subalgebra of $\mathfrak{g}$ corresponding to $G$.
We have the Levi decomposition 
$\mathfrak{p}=\mathfrak{a}_\mathfrak{p}\oplus \mathfrak{h}^\perp_\mathfrak{p}\oplus \mathfrak{n}_\mathfrak{p}$, where $\mathfrak{n}_\mathfrak{p}$ is the nilpotent 
radical, $\mathfrak{l}_\mathfrak{p}:=\mathfrak{a}_\mathfrak{p}\oplus \mathfrak{h}^\perp_\mathfrak{p}$ is the reductive Levi factor, $\mathfrak{a}_\mathfrak{p}$ is 
semi-simple and $\mathfrak{h}^\perp_\mathfrak{p}$ is the center of $\mathfrak{l}_\mathfrak{p}$. Sometimes we use the notation $\mathfrak{l}_G$, and similarly for the other 
factors, to emphasize the connection to $G$.

Putting the above things together, we can parameterize 
$\mathcal{S}$-subcategories in $\mathcal{O}$ using 
triples $(W,S,G)$, where $W$ is a Weyl group and $S$ and $G$
are two parabolic subgroups. We simply write $\mathcal{S}(W,S,G)$
for the corresponding $\mathcal{S}_{\lambda,G}$, which is well-defined,
up to equivalence.
Below we would like to address the following problem:

\begin{problem}\label{prob5}
Classify the blocks of $\mathcal{S}$-subcategories in $\mathcal{O}$,
up to equivalence.
\end{problem}

This accounts to determining for which triples
$(W,S,G)$ and $(W',S',G')$ we have $\mathcal{S}(W,S,G)\cong
\mathcal{S}(W',S',G')$.
Note that the main result of \cite{Co} solves a special case of
this problem, namely the case $G=G'=\{e\}$. In this paper we are not 
able to give a full solution, but we do solve a special case as well,
which can be seen as the ``opposite'' case to the one considered in \cite{Co}.

We say that $(W,S,G)$ and $(W',S',G')$ are {\em isomorphic} if there
is an isomorphism of Coxeter groups from $W$ to $W'$ which 
sends $S$ to $S'$ and $G$ to $G'$. Clearly, such an isomorphism induces
an equivalence of $\mathcal{S}(W,S,G)$ and $\mathcal{S}(W',S',G')$.

\subsection{Regular blocks}\label{s3.3}

Our classification result is for regular blocks, that is,
the case  $S=S'=\{e\}$. 
In this case we may take $\lambda=0$,
so that we work in the {\em principal block} $\mathcal{O}_0$. 
Then the simple
objects in $\mathcal{S}(W,\{e\},G)$ are naturally indexed by
the set $X_G$ of longest coset representatives in $G\setminus W$.
We denote by $B$ the unique, up to isomorphism, basic associative algebra such that
$B$-mod is equivalent to $\mathcal{S}(W,\{e\},G)$.

\begin{theorem}\label{thm6}
We have $\mathcal{S}(W,\{e\},G)\cong
\mathcal{S}(W',\{e\},G')$ if and only if $(W,\{e\},G)$
and $(W',\{e\},G')$ are isomorphic.
\end{theorem}

To prove this theorem, we will need more notation and some
preliminary results. Denote by $\mathtt{C}_W$ the {\em coinvariant
algebra} associated to $W$. It is defined as the quotient of
$\mathbb{C}[\mathfrak{h}]$ modulo the ideal generated by the
homogeneous $W$-invariant polynomials of non-zero degree.
In $\mathcal{O}_0$, the algebra $\mathtt{C}_W$ can be found as the
endomorphism algebra of the unique indecomposable projective
injective module $P_{w_0}$, here $w_0$ is the longest element of $W$,
see \cite[Endomorphismensatz]{So}. In particular, $\mathtt{C}_W$
is finite dimensional (of dimension $|W|$) 
and self-injective. We refer to \cite{Hi}
for more details on coinvariant algebras.

\begin{lemma}\label{lem7}
We have $\mathtt{C}_W\cong \mathtt{C}_{W'}$ if and only if
$W$ and $W'$ are isomorphic as Coxeter groups.
\end{lemma}

We note that the isomorphism of $\mathtt{C}_W\cong \mathtt{C}_{W'}$
in the formulation is just an isomorphism of algebras, not necessarily
of graded algebras.

\begin{proof}
The ``if'' part is obvious, so we prove the ``only if'' part.
To start with, we want to prove that $\mathtt{C}_W\cong \mathtt{C}_{W'}$
implies that the rank of $W$, i.e. the dimension of $\mathfrak{h}$,
equals the rank of $W'$.

Directly from the definition we have that $\mathtt{C}_W$ is positively
graded. For this proof, we choose the grading such that the generating
elements of $\mathfrak{h}$ have degree $1$ (note that this is different
from the natural grading coming from $\mathcal{O}$ in which the 
generators have degree $2$).

Since $\mathtt{C}_W$ is positively graded and generated by elements of 
degrees $0$ and $1$, by \cite[Proposition~2.4.1]{BGS}, the graded filtration of
$\mathtt{C}_W$ coincides with the radical filtration of the regular module
${}_{\mathtt{C}_W}\mathtt{C}_W$. Therefore the rank of $W$ can be 
recovered from $\mathtt{C}_W$ as the dimension of 
$\mathrm{Rad}({}_{\mathtt{C}_W}\mathtt{C}_W)/
\mathrm{Rad}(\mathrm{Rad}({}_{\mathtt{C}_W}\mathtt{C}_W))$.

Now, if we fix the rank of $W$, we can go through the finite list of finite
Weyl groups of this fixed rank and check that non-isomorphic Coxeter
groups of this fixed rank have different cardinalities. This is
clear if we compare them with each other from well-known lists: 
the Weyl group cardinality for type $A_n$ is $(n+1)!$, for type 
$B_n=C_n$ it is $2^nn!$ and for type $D_n$ it is $2^{n-1}n!$.
In the exceptional cases, we have the following cardinalities: for $E_6$ it is $72\cdot 6!$;
for $E_7$ it is $576\cdot 7!$; for $E_8$ it is $17280\cdot 8!$;
for $F_4$ it is $48\cdot 4!$ and in type $G_2$ it is $6\cdot 2!$.
The claim follows.
\end{proof}

\begin{lemma}\label{lem8}
The Bruhat order on $X_G$ coincides with the essential order for $B$.
\end{lemma}

\begin{proof}
This follows, for example, from the well-known fact that 
in category $\mathcal{O}$ $[\Delta_y:L_x]\neq 0$
if and only if $x\geq y$ with respect
to the Bruhat order, see e.g. \cite[Section~7]{Di}.
\end{proof}

\begin{proof}[Proof of Theorem~\ref{thm6}.]
The ``if'' part of the claim is obvious, so we prove the ``only if'' part.
Assume $\mathcal{S}(W,\{e\},G)\cong\mathcal{S}(W',\{e\},G')$.
Both categories $\mathcal{S}(W,\{e\},G)$ and $\mathcal{S}(W',\{e\},G')$
contain a unique, up to isomorphism, indecomposable projective-injective
module, namely the one whose index is the anti-dominant weight.
Therefore, any equivalence between these categories must match
these indecomposable projective-injective objects and induce
an isomorphism between the endomorphism algebras of these objects.
From \cite[Endomorphismensatz]{So} we thus get that there is an isomorphism
$\mathtt{C}_W\cong \mathtt{C}_{W'}$. From Lemma~\ref{lem7}
it follows that $W$ and $W'$ are isomorphic as Coxeter groups.

To complete the proof we just need to argue that there is an
isomorphism of Coxeter groups from $W$ to $W'$ which 
maps $G$ to $G'$. By Theorem~\ref{thm2}, we know that the
essential order on $X_G$ is unique and, by Lemma~\ref{lem8}, it 
coincides with the Bruhat order on $X_G$. Therefore any equivalence
$\mathcal{S}(W,\{e\},G)\cong\mathcal{S}(W',\{e\},G')$
induces an isomorphism of the posets $X_G$ and $X_{G'}$
with respect to the corresponding Bruhat orders.

Taking into account \cite[Theorem~1]{Co}, all such non-trivial
isomorphisms are classified in \cite[Theorem~2]{Co}.
The corresponding lists (1)-(4) in \cite[Theorem~2]{Co} are
(written as a pair (type of $W$, type of $G$)):
\begin{itemize}
\item $(A_{2n+1},A_{2n})\cong (B_{n+1},B_n)$, for $n\geq 2$;
\item $(B_n,A_{n-1})\cong (D_{n+1},A_n)$, for $n\geq 3$;
\item $(A_{3},A_{2})\cong (B_{2},A_1)$;
\item $(A_{5},A_{4})\cong (G_{1},A_1)\cong (B_{3},B_2)$.
\end{itemize}

We see that all non-trivial isomorphisms concern Weyl groups
of different ranks. Therefore an isomorphism $W$ to $W'$ in our case
can be chosen such that it maps $G$ to $G'$. This completes the proof.
\end{proof}

\subsection{The weighted poset invariant via the Bruhat order}\label{s3.7}

By construction, simple objects in $\mathcal{S}(W,S,G)$ are parameterized
by elements of the set $G\backslash W/S$
of double cosets (and they naturally correspond to the choice of
the longest representatives is such cosets). 
From Lemma~\ref{lem8}, it follows
that the essential order on $G\backslash W/S$
is exactly the restriction of the Bruhat order to the set of the 
longest coset representatives for cosets in $G\backslash W/S$.
We note that, by \cite{HS}, this coincides with the restriction of the Bruhat order 
$\leq $ to the set of the shortest representatives in cosets from 
$G\backslash W/S$.

The next lemma describes the weight function $\nu$ for 
$(G\backslash W/S,\leq)$ as defined in Subsection~\ref{s2.4}.

\begin{lemma}\label{lemn71}
The value of $\nu$ at a coset $\xi\in G\backslash W/S$ 
coincides with the number of right $S$-subcosets in $\xi$.
\end{lemma}

\begin{proof}
The standard module $\Delta_\xi$ is parabolically induced 
from an indecomposable 
projective-injective module $P(\mu)$ in the category $\mathcal{O}$ for the
Levi factor $\mathfrak{l}_G$ of the parabolic subalgebra $\mathfrak{p}$ 
corresponding to $G$, see \cite[Proposition~2.9]{MS}. 
In other words, there exists an anti-dominant weight $\mu$ for $\mathfrak{l}_G$
such that $\Delta_\xi\cong U(\mathfrak{g})\otimes_{U(\mathfrak{p})} P(\mu)$, 
where $P(\mu)$ is the projective cover of $L(\mu)$ 
in the category $\mathcal{O}$ for $\mathfrak{l}_G$.
	
Observe that the dimension of the endomorphism algebra of an indecomposable 
projective-injective module in category $\mathcal{O}$ equals the number of
isomorphism classes of simple objects in its block. This follows by
combining the following standard facts from the literature. First, indecomposable 
projective-injective  modules in $\mathcal{O}$ are exactly the projective 
covers of simple Verma modules by \cite{Ir}. Next, each Verma module has 
simple socle, this socle is a Verma module and none of the other 
simple  subquotients  of the original Verma module is a  Verma module,
see \cite[Section~7]{Di}. Finally, each indecomposable block of category 
$\mathcal{O}$ contains a unique simple Verma module (since each Weyl group
has a unique longest element) and hence, by the BGG reciprocity, 
the Verma flag of an indecomposable projective-injective module 
contains each Verma module from its block exactly once. 

Since, in the situation of our lemma, the simple objects in the block of $\mathcal{O}$
which we consider correspond exactly to  the right $S$-subcosets in $\xi$,
we obtain that the dimension of the endomorphism algebra of $P(\mu)$
is given by the number of right $S$-subcosets in $\xi$.

The parabolic induction provides an embedding of the endomorphism algebra of
$P(\mu)$ into the endomorphism algebra of $\Delta_\xi$. That this embedding is 
an isomorphism, is obtained, using adjunction, from the following:
\begin{displaymath}
\mathrm{Hom}_{\mathfrak{g}}(\Delta_\xi,\Delta_\xi)\cong
\mathrm{Hom}_{\mathfrak{p}}(P(\mu),\mathrm{Res}^{\mathfrak{g}}_{\mathfrak{p}}\Delta_\xi)\cong
\mathrm{Hom}_{\mathfrak{p}}(P(\mu),P(\mu)),
\end{displaymath}
Here, to justify the second isomorphism, we note that the elements of 
$\mathfrak{h}^\perp_\mathfrak{p}$ act on  $P(\mu)$ as scalars. Furthermore, 
the action of $\mathfrak{h}^\perp_\mathfrak{p}$ on 
$\mathrm{Res}^{\mathfrak{g}}_{\mathfrak{p}}\Delta_\xi$ is semisimple
with $P(\mu)$ being an isotypic component. Since $\mathfrak{p}$
contains $\mathfrak{h}^\perp_\mathfrak{p}$, the image of any 
$\mathfrak{p}$-homomorphism from $P(\mu)$ to 
$\mathrm{Res}^{\mathfrak{g}}_{\mathfrak{p}}\Delta_\xi$
belongs to $P(\mu)$.

This completes the proof of the lemma.
\end{proof}

\subsection{Some derived equivalences in type $A$}\label{s3.8}

In type $A$, that is when $W=S_n$, a parabolic subgroup 
$W'$ of $W$ defines naturally a partition of $n$ which bookkeeps
the sizes of the connected blocks in $\{1,2,\dots,n\}$,
where $i$ and $i+1$ are connected if $(i,i+1)\in W'$.

In the next two propositions, when we say that two
module categories are derived equivalent we mean that the corresponding
derived categories of right bounded complexes are equivalent.
All derived equivalences which we construct are given by (generalized)
tilting modules. Hence our derived equivalence induces an equivalence of 
the categories of perfect complexes.

\begin{proposition}\label{propnew72}
Assume that $W$ is the symmetric group $S_n$ and
that $(W,S,G)$ and $(W,S',G)$ are such that  
$S$ and $S'$ correspond to the same partition of $n$.
Then the category $\mathcal{S}(W,S,G)$ is 
derived equivalent to the category $\mathcal{S}(W,S',G)$.
\end{proposition}

\begin{proof}
We have the derived equivalence between 
$\mathcal{S}(W,S,\{e\})$ and $\mathcal{S}(W,S',\{e\})$ given by
\cite[Theorem~B]{CM2}. 
By construction in \cite{CM2}, this equivalence 
is defined in terms of projective and shuffling functors and such functors 
send $\mathcal{S}(W,S,G)$ to $\mathcal{S}(W,S',G)$. 
Since the inclusion of the homotopy category of projective
objects in $\mathcal{S}(W,S,G)$ into the homotopy category of 
projective objects in $\mathcal{S}(W,S,\{e\})$ is full and faithful,
the above derived equivalence between 
$\mathcal{S}(W,S,\{e\})$ and $\mathcal{S}(W,S',\{e\})$ 
restricts to a  derived equivalence between
$\mathcal{S}(W,S,G)$ and $\mathcal{S}(W,S',G)$.
\end{proof}

\begin{proposition}\label{propnew73}
Assume that $W$ is the symmetric group $S_n$ and
that $(W,S,G)$ and $(W,S,G')$ are such that  
$G$ and $G'$ correspond to the same partition of $n$.
Then the category $\mathcal{S}(W,S,G)$ is 
derived equivalent to the category $\mathcal{S}(W,S,G')$.
\end{proposition}

\begin{proof}
This is very similar to the proof of \cite[Theorem~B]{CM2},
but much easier as we can use twisting functors from \cite{AS}
which, in particular, commute with projective functors.
It is enough to assume that $G$ and $G'$ differ by 
permuting two neighbor parts. 
We let $w_0^1$ denote the longest
element of the parabolic subgroup corresponding to the first 
of these parts for $G$ and we let
$w_0^2$ denote the longest
element of the parabolic subgroup corresponding to the second part. 
Note that $w_0^1$ and $w_0^2$ commute since they are permutations 
of sets of elements that do not intersect.
Finally, we let  $w_0^{1,2}$ denote the longest
element of the parabolic subgroup corresponding to the merging
of the two parts. Note that the parts of $G'$ are obtained from
the parts of $G$ by conjugation with $w_0^{1,2}$.
We will use various properties of 
$\mathcal{S}$-subcategories established in \cite{FKM,MS}
and also various properties of twisting functor from \cite{AS}.

First we claim that the twisting functor
$\top_{w_0^{1,2}}$ induces a derived equivalence between 
$\mathcal{S}(W,\{e\},G)$ and $\mathcal{S}(W,\{e\},G')$.
For this we can realize both subcategories inside $\mathcal{O}_0$.
As usual, $w_0^G$ denotes the longest element in $G$. 
Then $w_0^G$ is the product of the longest element in each 
parabolic subgroup corresponding to a connected block in $G$.
Set $w_0^r:=w_0^Gw_0^1w_0^2$, then $w_0^G=w_0^rw_0^1w_0^2$ is reduced. 
Note that $w_0^r$ commutes with $w_0^1$, $w_0^2$ and $w_0^{1,2}$.
Projectives in $\mathcal{S}(W,\{e\},G)$ are obtained by applying 
projective functors to the dominant projective (thus standard)
module $P_{w_0^G}=\theta_{w_0^G}\Delta_e$, where $\theta_{w}$ is the 
indecomposable projective functor uniquely defined by 
$\theta_w\Delta_\lambda=P(w\cdot \lambda)$, see \cite{BG}.
Applying $\top_{w_0^{1,2}}$ to this module, we have
\begin{displaymath}
\top_{w_0^{1,2}}P_{w_0^G}=\top_{w_0^{1,2}}\theta_{w_0^G}\Delta_e=
\theta_{w_0^G}\top_{w_0^{1,2}}\Delta_e=
\theta_{w_0^G}\Delta_{w_0^{1,2}}.
\end{displaymath}
The module $P_{w_0^G}$ has a Verma flag with Verma subquotients
of the form $\Delta_w$, where $e\leq w\leq w_0^G$. As 
$\theta_{w_0^G}$ acts on the indices on the right,
the module $\theta_{w_0^G}\Delta_{w_0^{1,2}}$ has a Verma flag 
with Verma subquotients of the form $\Delta_w$, where  $w$
belongs to the coset $G'w_0^{1,2}=w_0^{1,2}G$. 
Note that
\begin{displaymath}
w_0^{1,2}w_0^G=w_0^{1,2}w_0^1w_0^2w_0^r=
(w_0^{1,2}w_0^1w_0^{1,2})(w_0^{1,2}w_0^2w_0^{1,2})w_0^rw_0^{1,2}=w_0^{G'}w_0^{1,2}.
\end{displaymath}
Let $q$ be the longest element in the coset $G'w_0^{1,2}$.
The module $\theta_{w_0^G}\Delta_{w_0^{1,2}}$ has simple
top $L_{q}$ and thus is a quotient of the indecomposable projective module
$P_q$. From the Verma flag of $\theta_{w_0^G}\Delta_{w_0^{1,2}}$ described above,
we see that the kernel of the projection $P_q\tto \theta_{w_0^G}\Delta_{w_0^{1,2}}$
is generated by all Verma subquotients of $P_q$ which do not appear in the
Verma flag of $\theta_{w_0^G}\Delta_{w_0^{1,2}}$. 
This implies that $\top_{w_0^{1,2}}P_{w_0^G}=\theta_{w_0^G}\Delta_{w_0^{1,2}}$ 
is a standard
object in $\mathcal{S}(W,\{e\},G')$.

Applying projective functors and using that they commute with
twisting functors, we obtain that $\top_{w_0^{1,2}}$ sends projectives
in $\mathcal{S}(W,\{e\},G)$ to modules with standard filtrations
in $\mathcal{S}(W,\{e\},G')$. The inclusion of the homotopy category of projectives in $\mathcal{S}(W,\{e\},G)$ into the homotopy category of 
projectives in $\mathcal{O}_0$ is full and faithful. Since
$\top_{w_0^{1,2}}$ is a derived equivalence for $\mathcal{O}$,
it thus follows that the image of a projective generator from 
$\mathcal{S}(W,\{e\},G)$ under $\top_{w_0^{1,2}}$ is a generalized tilting
module in $\mathcal{S}(W,\{e\},G')$.
This implies that 
$\top_{w_0^{1,2}}$ induces a derived equivalence between 
$\mathcal{S}(W,\{e\},G)$ and $\mathcal{S}(W,\{e\},G')$.

Finally, as twisting functors
commute with projective functors (in particular,  with translations
to and from the walls), the derived equivalence constructed
above induces a derived equivalence
between $\mathcal{S}(W,S,G)$ and $\mathcal{S}(W,S,G')$.
\end{proof}

\subsection{Small rank cases}\label{s3.9}

The weighted poset $G\backslash W/S$
with the weight function given by Lemma~\ref{lemn71} gives a 
strong invariant for telling different blocks $\mathcal{S}(W,S,G)$ apart.
This allows us to deal with many small rank cases.
The poset itself can be computed using computer and suitable software.
In Section~\ref{s5} one can find the results of our computations using
Sage. Below is a short analysis of these computations.

In type $A_2$, which can be found in Figure~\ref{fig3}, there is the
obvious redundancy given by the symmetry of the root system.
Also, all  blocks in the last row are obviously equivalent
(to the category of vector spaces). Taking this into account 
(i.e. identifying the blocks obtained from each other by the involution
swapping $s_1$ and $s_2$ and collapsing the last row to a single element), 
we get an irredundant classification of blocks in type $A_2$, up to equivalence.

Similarly one can deal with the types $B_2$ and $G_2$, which can be found
in Figure~\ref{fig3}, respectively.

Type $A_3$, which can be found in Figure~\ref{fig4}, is more  interesting.
There are two places  for which the obvious redundancy as above is not
enough. These are indicated by the magenta and the orange colors in the figure.
The two orange cases, $(A_3,\langle s_1 \rangle, A_3)$ and $(A_3,\langle s_2 \rangle, A_3)$, correspond to module categories over 
the algebras of all $(1,2)$- and all $(2,3)$-invariants in the
coinvariant algebra of $S_4$, respectively. 
As these algebras are naturally isomorphic
(since $(1,2)$  and $(2,3)$ are conjugate in $S_4$ { and hence the action on
the coinvariant algebra of an element which conjugates $(1,2)$ to $(2,3)$ gives rise to an isomorphism
between the corresponding algebras of $(1,2)$-  and $(2,3)$-invariants}), we obtain that the 
orange cases are equivalent.
We do not know whether the magenta cases, $(A_3,\langle s_1 \rangle, \langle s_2 \rangle)$ and $(A_3,\langle s_2 \rangle, \langle s_1 \rangle)$, are equivalent or not,
we suspect that they are.
They are certainly derived equivalent since they both are derived equivalent 
to the teal case, $(A_3,\langle s_2 \rangle, \langle s_2 \rangle)$. The magenta case in the same column with the teal  case
is  derived equivalent to the teal case by Proposition~\ref{propnew72}.
The magenta case in the same row with the teal  case
is  derived equivalent to the teal case by Proposition~\ref{propnew73}. 

Type $B_3$, which can be found in Figure~\ref{fig5},
further supports the observation that, with a  small number of exceptions,  
different blocks of $\mathcal{S}(W,S,G)$ will not be equivalent.
The only cases of suspected, but not obvious by other kinds of symmetries, 
equivalences are indicated by the same color.

In general, it is not true that $\mathcal{S}(W,S,G)$
and $\mathcal{S}(W',S',G')$ are equivalent if and only if 
their weighted posets are isomorphic. A very degenerate counterexample
would be the one element poset of weight $6$. This can be realized,
on the one hand, by the coinvariant algebra of type $A_2$
and, on the other hand,  by the algebra of $A_1$-invariants
in the coinvariant algebra of type $G_2$, see Figure~\ref{fig3}. These two algebras
are not isomorphic (the first one  has two generators while  the
second one has only one generator). Therefore, to tell the blocks
$\mathcal{S}(W,S,G)$ apart, a more subtle invariant is necessary.
A good starting point would be:

\begin{problem}\label{conjnew73}
{\rm
Classify, up to isomorphism, the algebras of parabolic invariants 
in all coinvariant algebras
(alternatively, all $\mathcal{S}(W,S,G)$
having exactly one isomorphism class of simple objects).
}
\end{problem}

Figures~\ref{fig3}--\ref{fig5} suggest the following conjecture which, for the moment, is completely open:

\begin{conjecture}\label{conjnew75}
{\rm
The categories $\mathcal{S}(W,S,S)$
and $\mathcal{S}(W',S',S')$ are equivalent if and only if 
their weighted posets are isomorphic. 
}
\end{conjecture}

\section{Homological properties of $\mathcal{S}(W,\{e\},G)$}\label{s4}

\subsection{Objects of finite projective dimension}\label{s4.1}

We take this opportunity to describe various homological
invariants for $\mathcal{S}(W,\{e\},G)$, in particular, improving
on the main result of \cite{MO}. The latter says that the finitistic
dimension of $\mathcal{S}(W,\{e\},G)$ equals twice the projective
dimension of the characteristic tilting module. This characteristic
tilting module is an additive generator of the full subcategory of
$\mathcal{S}(W,\{e\},G)$ consisting of all modules which have
both standard and proper costandard filtrations. 
We improve this result
by giving an explicit formula for the projective dimension of this
tilting module, reducing it to the known results in category $\mathcal{O}$
from \cite{Ma1,Ma2}.

Recall that $X_G$ is the set of all longest coset representatives in
$G\setminus W$ and it indexes simple objects in $\mathcal{S}(W,\{e\},G)$
which we consider as the quotient category of $\mathcal{O}_0$.
Recall that we  denote by $\mathfrak{a}_G$ the semi-simple part of the
parabolic subalgebra $\mathfrak{p}$ of $\mathfrak{g}$ corresponding to
$G$.
Each $M$ in the category $\mathcal{O}$ for $\mathfrak{g}$, when 
considered as an $\mathfrak{a}_G$-module, is a (possibly infinite)
direct sum of modules from the category $\mathcal{O}$ for $\mathfrak{a}_G$.
We say that $M$ is {\em admissible} provided that all summands of
this latter decomposition are projective-injective
(in the category $\mathcal{O}$ for $\mathfrak{a}_G$).
Our first observation is the following:

\begin{lemma}\label{lem11}
Let $M\in \mathcal{O}$ be admissible. Assume that
$\mathrm{Ext}_{\mathcal{O}}^i(M,L_w)\neq 0$, for some
$w\in W$ and some $i\geq 0$. Then $w\in X_G$.
\end{lemma}

\begin{proof}
For a simple reflection $s$ in $G$, consider the corresponding
twisting functor $\top_s$, see \cite{AS}. The element $s$ corresponds to a
simple root in $\mathfrak{a}_G$, say $\alpha$. 

We claim that $\top_sM\cong M$, more precisely, that the 
natural map $\top_sM\to M$ given by \cite[Theorem~4]{KM} is an isomorphism. 
Since $M$ is admissible, when restricted to
$\mathfrak{a}_G$, the module $M$ is a direct sum of projective modules.
Therefore the map  $\top_sM\to M$ is injective by \cite[Theorem~4]{KM}. 

On the other hand, since $M$ is admissible, when restricted to
$\mathfrak{a}_G$, the module $M$ is a direct sum of injective modules.
Therefore the action of $\mathfrak{g}_{-\alpha}$ is injective
on any top constituent of $M$. This implies that the cokernel of
the natural map $\top_sM\to M$ is zero and thus $\top_sM\cong M$.

Now let $w\not\in X_G$. Then there is a simple reflection 
$s$ as above such that $sw>w$. Let $G_s$ denote the right adjoint of
$\top_s$. Then $G_s\cong\star \top_s\star$ by \cite[Theorem~4.1]{AS}.
By \cite[Theorem~6.1]{AS}, {we have} $\mathcal{L}_1\top_sM=0$, since 
$\mathfrak{g}_{-\alpha}$ acts injectively on $M$, 
and thus $\top_sM\cong\mathcal{L}\top_sM$. 
Combining \cite[Theorem~4.1]{AS} and \cite[Theorem~6.1]{AS}, 
we also get $\mathcal{R}G_sL_w\cong L_w[-1]$.
Using the above, we compute:
\begin{displaymath}
\begin{array}{rcl}
\mathrm{Ext}^i_{\mathcal{O}}(M,L_w)&=&
\mathrm{Hom}_{\mathcal{D}^b(\mathcal{O})}(M,L_w[i])\\ 
&=&\mathrm{Hom}_{\mathcal{D}^b(\mathcal{O})}(\top_sM,L_w[i])\\ 
&=&\mathrm{Hom}_{\mathcal{D}^b(\mathcal{O})}(\mathcal{L}\top_sM,L_w[i])\\ 
&=&\mathrm{Hom}_{\mathcal{D}^b(\mathcal{O})}(M,\mathcal{R}G_sL_w[i])\\ 
&=&\mathrm{Hom}_{\mathcal{D}^b(\mathcal{O})}(M,L_w[i-1]).
\end{array}
\end{displaymath}
Applying this procedure $i+1$ times we get a homomorphism in
$\mathcal{D}^b(\mathcal{O})$ from a module to a module shifted to the right.
This is certainly zero. The claim of the lemma follows.
\end{proof}

\begin{theorem}\label{thm12}
An object $N\in \mathcal{S}(W,\{e\},G)$
has finite projective dimension if and only if
there is an admissible $M\in\mathcal{O}_0$ such that
$N\cong\Phi(M)$, where $\Phi$ is the quotient functor. Moreover, the minimal projective resolution
of $N$ in $\mathcal{S}(W,\{e\},G)$ is obtained from a
minimal projective resolution of $M$ in $\mathcal{O}_0$
by applying $\Phi$ and the projective dimensions of
$N$ (in $\mathcal{S}(W,\{e\},G)$) and $M$ (in $\mathcal{O}_0$) coincide.
\end{theorem}

\begin{proof}
Assume that $M\in \mathcal{O}_0$ is admissible. Let
$\mathcal{P}_{\bullet}$ be a minimal projective resolution
of $M$. By Lemma~\ref{lem11}, the only indecomposable
projectives appearing in $\mathcal{P}_{\bullet}$ are
$P_w$, where $w\in X_G$. 
The images of these projective modules
under the exact functor $\Phi$ are projective in 
$\mathcal{S}(W,\{e\},G)$, moreover, $\Phi$ preserves the property
of such projectives to be pairwise non-isomorphic. 
Therefore
$\Phi(\mathcal{P}_{\bullet})$ is a minimal projective resolution
of $\Phi(M)$.

Conversely, assume that $N\in \mathcal{S}(W,\{e\},G)$ has a finite
projective resolution. Since any {indecomposable projective 
in $ \mathcal{S}(W,\{e\},G)$ is, up to isomorphism, 
the image under $\Phi$  of some $P_w$, where $w\in X_G$,}   and since
$\Phi$ is full and faithful on such projectives, 
we may assume that a
minimal projective resolution of $N$ has the form 
$\Phi(\mathcal{P}_{\bullet})$, for some complex 
$\mathcal{P}_{\bullet}$ of projective modules in $\mathcal{O}_0$,
in which each indecomposable projective has the form 
$P_w$, for some $w\in X_G$. 
We claim that 
$\mathcal{P}_{\bullet}$ has no homology in $\mathcal{O}_0$
outside of the homological position $0$, and that the latter
homology is an admissible module $M$ such that $\Phi(M)\cong N$
(this isomorphism is obvious by construction).

To this end, we first note that each $P_w$, for $w\in X_G$,
is admissible. Indeed, if $w_0^G$ is the longest element of
$G$, then $P_{w_0^G}$ is parabolically induced from its
$\mathfrak{l}_G$-direct summand on which the nilradical
of the parabolic subalgebra acts as $0$. 
This direct summand
is the projective module in the category $\mathcal{O}$ for
$\mathfrak{l}_G$ corresponding to the anti-dominant element
$w_0^G$, hence injective.
Since parabolic induction, when
restricted to $\mathfrak{l}_G$,  is an infinite direct sum of 
tensorings with finite dimensional
$\mathfrak{l}_G$-modules {and such tensorings preserve
both injectivity and projectivity}, it follows that $P_{w_0^G}$ is admissible.

All other $P_w$, for  $w\in X_G$, are summands of images of
$P_{w_0^G}$ under projective functors on $\mathcal{O}_0$,
see \cite{BG}. 
These functors are summands of tensorings with 
finite dimensional $\mathfrak{g}$-modules. 
It follows that all 
$P_w$, for  $w\in X_G$, are admissible.

Next we note that the cokernel of an injection
between two admissible modules is admissible
(since both modules are direct sums of projective-injective
objects when restricted to $\mathfrak{l}_G$ and thus the same property must
be inherited by the cokernel).

By our assumptions, $\Phi(\mathcal{P}_{\bullet})$ is exact
at all non-zero homological positions.
Therefore the only possible simples occurring in the homology
of $\mathcal{P}_{\bullet}$ in non-zero positions are
$L_w$, for $w\not \in X_G$. 
For each such $w$, there exists a
simple reflection $s\in G$ such that, for the corresponding
simple root $\alpha$, the action of $\mathfrak{g}_{-\alpha}$
on $L_w$ is locally nilpotent. 
Now let us take the
leftmost non-zero homology of $\mathcal{P}_{\bullet}$. By the
previous paragraph, applied inductively, this homology 
must be a submodule of an admissible module. 
However, 
as we have already mentioned before, all $\mathfrak{g}_{-\alpha}$
as above act injectively on admissible modules, a contradiction.

Consequently, all homologies in non-zero homological position of $\mathcal{P}_{\bullet}$
are zero and the argument in the previous paragraph also proves that
in this case the zero homology of $\mathcal{P}_{\bullet}$ is admissible.
This completes the proof.
\end{proof}

\subsection{The finitistic dimension}\label{s4.2}

Theorem~\ref{thm12} reduces computation of projective dimension
in $\mathcal{S}(W,\{e\},G)$ to the same question for
$\mathcal{O}_0$. 
For the latter category, projective dimensions
of structural modules are determined in \cite{Ma1,Ma2},
see also \cite{CM}. We first record the following result which
determines the finitistic dimension of $\mathcal{S}(W,\{e\},G)$
explicitly, thus strengthening the main result of \cite{MO}.
Let $\mathbf{a}:W\to \mathbb{Z}_{\geq 0}$  denote Lusztig's
$\mathbf{a}$-function, see \cite{Lu}.

\begin{corollary}\label{cor14}
The finitistic dimension of $\mathcal{S}(W,\{e\},G)$ 
equals $2\mathbf{a}(w_0^{G}w_0)$.
\end{corollary}

\begin{proof}
By the main result of \cite{MO}, the finitistic dimension of 
$\mathcal{S}(W,\{e\},G)$ equals twice the projective dimension
of the characteristic tilting module. By \cite{FKM}, the indecomposable
tilting modules in $\mathcal{S}(W,\{e\},G)$ are exactly
$\Phi(T_w)$, where $w$ is a shortest coset representative
in $G\setminus W$. 

{
Observe that all such $T_{w}$ are admissible by essentially the
same arguments as for projective modules: the module $T_{w_0^{G}w_0}$
is parabolically induced from a projective-injective 
$\mathfrak{l}_G$-module and all other $T_w$ are obtained from
$T_{w_0^{G}w_0}$ by applying projective functors. 
}

By the main result of \cite{Ma2}, the maximal
projective dimension for such $T_w$ is achieved for $w=w_0^{G}w_0$
and has the value $\mathbf{a}(w_0^{G}w_0)$. The claim now
follows from Theorem~\ref{thm12}.
\end{proof}

\subsection{Various explicit projective dimensions}\label{s4.3}

For $w\in X_G$, we denote by $\Delta^{\mathcal{S}}_w$ the corresponding
standard object in $\mathcal{S}(W,\{e\},G)$, by $\nabla^{\mathcal{S}}_w$
the corresponding costandard object, by $T^{\mathcal{S}}_w$ the corresponding
indecomposable tilting (and cotilting) object and by 
$I^{\mathcal{S}}_w$ the corresponding
indecomposable injective object. For the homological invariants of
indecomposable structural modules, we have the following result.

\begin{corollary}\label{cor15}
For  $w\in X_G$, we have
\begin{displaymath}
\mathrm{proj.dim}(T^{\mathcal{S}}_w)=\mathbf{a}(w_0^{G}w)\quad\text{ and }\quad
\mathrm{proj.dim}(I^{\mathcal{S}}_w)=2\mathbf{a}(w_0w).
\end{displaymath}
\end{corollary}

\begin{proof}
As noted above, $T^{\mathcal{S}}_w=\Phi(T_{w_0^{G}w})$.
Similarly to the proof of Corollary~\ref{cor14}, the equality
$\mathrm{proj.dim}(T^{\mathcal{S}}_w)=\mathbf{a}(w_0^{G}w)$
follows directly from Theorem~\ref{thm12} and the main results
of \cite{Ma1,Ma2}.

Any costandard object has a finite resolution by cotilting objects.
Since all cotilting objects are tilting and hence
are images under $\Phi$ of admissible
objects, the same argument as in the
proof of Theorem~\ref{thm12} implies that all costandard objects are
images under $\Phi$ of admissible objects. 
As any injective object has 
a filtration by costandard objects, it follows that any injective object 
is an image under $\Phi$ of an admissible object. 
Hence all
injectives in $\mathcal{S}(W,\{e\},G)$ have finite projective dimension 
by Theorem~\ref{thm12}.
Further, since $\Delta^{\mathcal{S}}_{w_0^G}=P^{\mathcal{S}}_{w_0^G}$,
applying $\star$ (cf. \cite[Section~2]{MS}) we obtain that
$\nabla^{\mathcal{S}}_{w_0^G}=I^{\mathcal{S}}_{w_0^G}=
\Phi(I_{w_0^G})$, where $I_{w_0^G}$ is admissible. 
Similarly,
$I^{\mathcal{S}}_w=\Phi(I_{w})$, for every $w\in X_G$, and each such 
$I_{w}$ is admissible 
{(because it is a summand of the image of the admissible module
$I_{w_0^G}$ under a projective functor)}. 
Hence the desired equality
$\mathrm{proj.dim}(I^{\mathcal{S}}_w)=2\mathbf{a}(w_0w)$
follows directly from Theorem~\ref{thm12} and the main results
of \cite{Ma1,Ma2}.
\end{proof}

Recall from \cite[Page~1084]{CM} the function $\mathbf{d}_{\lambda}$, where
$\lambda$ is an integral dominant weight, which assigns to $w\in W$ the
projective dimension of the (potentially singular) Verma module $\Delta(w\cdot \lambda)$.
This is one of {the} two functions which were used in \cite{CM} to describe homological
invariants in singular blocks of parabolic category $\mathcal{O}$.

\begin{proposition}\label{prop19}
For $w\in X_G$, we have $\mathrm{proj.dim}(\Delta^{\mathcal{S}}_w)=\mathbf{d}_{\lambda}(w^{-1})$,
where $\lambda$ is such that its dot-stabilizer coincides with $G$.
\end{proposition}

\begin{proof}
This result follows by collecting various statements in the literature. In the natural
$\mathbb{Z}$-grading of category $\mathcal{O}$ introduced in \cite{So}, the module
$\Delta^{\mathcal{S}}_w$ has a linear projective resolution by \cite[Corollary~8.4(i)]{Ma3}.
Combining the Ringel self-duality of $\mathcal{S}(W,\{e\},G)$, 
see \cite[Theorem~3]{FKM},
with \cite[{Theorem~3.3} and Corollary~8.7]{Ma3}, 
we can identify the {linear} projective resolution of
$\Delta^{\mathcal{S}}_w$ with the maximal quotient $N_{w}$ of $\Delta_{w_0w^{-1}}$ having as composition
subquotients only those $L_x$ for which $x$ is a minimal coset representative in
$W/G$. 
Hence the projective dimension of $\Delta^{\mathcal{S}}_w$ coincides with the
Loewy (or graded) length of $N_{w}$. 
Soergel's self-equivalence for the full subcategory of 
$\mathcal{O}_0$ consisting of modules on which the action of the center of $U(\mathfrak{g})$
is scalar, see \cite{So0}, matches $N_{w}$ with a parabolic Verma module
(a maximal quotient of $\Delta_{ww_0}$ having as composition
subquotients only those $L_x$ for which $x$ is a minimal coset representative in
$G\setminus W$). 
Therefore the
claim of the proposition follows from \cite[Table~2]{CM}.
\end{proof}

An interesting problem seems to be:

\begin{problem}\label{prob17}
Determine the projective dimension of $\nabla^{\mathcal{S}}_w$,
where $w\in X_G$.
\end{problem}

The following example suggests that Problem~\ref{prob17} is non-trivial.

\begin{example}\label{ex18}
{\rm 
Let $\mathfrak{g}=\mathfrak{sl}_3$ with $W=\{e,s,t,st,ts,w_0\}$
and $G=\{e,s\}$. Then $X_G=\{s,st,w_0\}$. We have 
$\Delta^{\mathcal{S}}_{w_0}=\Phi(T_{ts})$. 
The module $T_{ts}$ admits
a short exact sequence 
$\Delta_{ts}\hookrightarrow T_{ts}\tto \Delta_{w_0}$.
Note that $\mathrm{proj.dim}(\Delta_{w_0})=3$,
$\mathrm{proj.dim}(\Delta_{ts})=2$ while 
$\mathrm{proj.dim}(T_{ts})=\mathrm{proj.dim}(\Delta^{\mathcal{S}}_{w_0})=1$
(in particular, it is strictly smaller than the projective dimension
of the individual components of the standard filtration).
For costandard filtrations it is even worse.
One can note that the module $T_{ts}$ admits
a short exact sequence 
$\nabla_{w_0}\hookrightarrow T_{ts}\tto \nabla_{ts}$
and that $\mathrm{proj.dim}(\nabla_{w_0})=3$ and
$\mathrm{proj.dim}(\nabla_{ts})=4$. At the same time, we have
$\nabla^{\mathcal{S}}_{w_0}=\Phi(T_{ts})$.
} 
\end{example}

\section{Weighted essential posets in types $A_2$, $B_2$, $G_2$, $A_3$ and $B_3$}\label{s5}

In Figures~\ref{fig3}--\ref{fig5} one can find lists of the 
posets $G\backslash W/S$ 
(with respect to the Bruhat order on longest representatives) 
in types $A_2$, $B_2$, $G_2$, $A_3$ and $B_3$. 
The picture gives the Hasse diagram of the poset 
with each vertex represented by its weight  as given by Lemma~\ref{lemn71}.

\begin{figure}
\centering
\resizebox{15.5cm}{!}{
\begin{tabular}{c | c | c | c | c}
 & \(G=\{e\}\) & \(G=\langle s_1 \rangle\) & \(G=\langle s_2 \rangle\) & \(G=\langle s_1, s_2 \rangle\)
\\
\end{tabular}
\hspace{2cm}

\begin{tabular}{c | c | c | c | c}
 & \(G=\{e\}\)& \(G=\langle s_1 \rangle\)& \(G=\langle s_2 \rangle\)& \(G=\langle s_1, s_2 \rangle\)\\
\end{tabular}
\hspace{2cm}

\begin{tabular}{c | c | c | c | c}
 & \(G=\{e\}\) & \(G=\langle s_1 \rangle\) & \(G=\langle s_2 \rangle\) & \(G=\langle s_1, s_2 \rangle\)
\\
\end{tabular}
}
\caption{Weighted posets $G\backslash W/S$ in types $A_2$, $B_2$ and $G_2$}\label{fig3}
\end{figure}

%%%%%%%%%%%%%%%%%%%%%%%%%%%%%%%%%%%%%%%%%%%%%%%%%%%%

\begin{figure}
\centering
\resizebox{!}{11.5cm}{
\begin{tabular}{c | c | c | c | c | c | c | c | c}
 & \(G=\{e\}\)& \(G=\langle s_1 \rangle\)& \(G=\langle s_2 \rangle\)& \(G=\langle s_3 \rangle\)& \(G=\langle s_1, s_2 \rangle\)& \(G=\langle s_1, s_3 \rangle\)& \(G=\langle s_2, s_3 \rangle\)& \(G=\langle s_1, s_2, s_3 \rangle\)  \\
\end{tabular}

}
\caption{Weighted posets $G\backslash W/S$ in type $A_3$}\label{fig4}
\end{figure}

%%%%%%%%%%%%%%%%%%%%%%%%%%%%%%%%%%%%%%%%%%%%%%%%%%%%

\begin{figure}
\centering
\resizebox{!}{11.5cm}{

\begin{tabular}{c | c | c | c | c | c | c | c | c}
 & \(G=\{e\}\)& \(G=\langle s_1 \rangle\)& \(G=\langle s_2 \rangle\)& \(G=\langle s_3 \rangle\)& \(G=\langle s_1, s_2 \rangle\)& \(G=\langle s_1, s_3 \rangle\)& \(G=\langle s_2, s_3 \rangle\)& \(G=\langle s_1, s_2, s_3 \rangle\)\\
\end{tabular}

}
\caption{Weighted posets $G\backslash W/S$ in type $B_3$}\label{fig5}
\end{figure}

\vspace{5mm}

\noindent
Department of Mathematics, Uppsala University, Box. 480,
SE-75106, Uppsala, SWEDEN, 

\noindent
VM email: {\tt mazor\symbol{64}math.uu.se}\\
EPW email: {\tt elin.persson.westin\symbol{64}math.uu.se}

\end{document}